\documentclass[12pt,reqno]{amsart}

\usepackage{amsfonts}
\usepackage{amsmath}
\usepackage{amssymb}
\usepackage{hyperref,thmtools}
\usepackage{amsthm}
\usepackage{verbatim}
\usepackage{ulem}
\usepackage{cancel}
\usepackage{tikz, tikz-cd}
\usetikzlibrary{calc}
\usepackage{ytableau}
\usepackage{subcaption}

\theoremstyle{plain}
\newtheorem{theorem}{Theorem}[section]
\newtheorem{proposition}[theorem]{Proposition}
\newtheorem{lemma}[theorem]{Lemma}
\newtheorem{corollary}[theorem]{Corollary}

\theoremstyle{definition}

\theoremstyle{remark}

\theoremstyle{example}
\newtheorem{example}[theorem]{Example}

\numberwithin{equation}{section}

\usepackage[]{cite}

\newcommand\QQ{\mathbb{Q}}

\newcommand\ZZ{\mathbb{Z}}

\newcommand\Par{\mathrm{Par}}

\newcommand\up{\mathrm{up}}

\newcommand\Inv{\mathrm{Inv}}
\newcommand\Quinv{\mathrm{Quinv}}
\newcommand\inv{\mathrm{inv}}

\newcommand\arm{\mathrm{arm}}
\newcommand\leg{\mathrm{leg}}
\newcommand\maj{\mathrm{maj}}

\newcommand{\rev}{\mathrm{rev}}

\newcommand{\wt}{\mathrm{wt}}
\newcommand{\Area}{\mathrm{Area}}

\newcommand{\coleg}{\mathrm{coleg}}
\newcommand\dinv{\mathrm{dinv}}
\newcommand\WP{\mathcal{WP}}
\newcommand{\len}{\mathrm{len}}

\newcommand{\NEpath}[4]{
	\fill[white!25]  (#1) rectangle +(#2,#3);
	\fill[fill=white]
	(#1)
	\foreach \dir in {#4}{
		\ifnum\dir=0
		-- ++(1,0)
		\else
		-- ++(0,1)
		\fi
	} |- (#1);
	\draw[help lines] (#1) grid +(#2,#3);
	\draw[dashed] (#1) -- +(#3,#3);
	\foreach \x in {1,...,#2} {
		\node[below,gray] at ($(#1) + (\x-0.5, -0.3)$) {\x};
	}
	
	\foreach \y in {1,...,#3} {
		\node[left,gray] at ($(#1) + (-0.3, \y-0.5)$) {\y};
	}
	\coordinate (prev) at (#1);
	\foreach \dir in {#4}{
		\ifnum\dir=0
		\coordinate (dep) at (1,0);
		\else
		\coordinate (dep) at (0,1);
		\fi
		\draw[line width=2pt,-stealth] (prev) -- ++(dep) coordinate (prev);
	};
}

\title[]{Equating Inv-Quinv formulas for the \\$q$-Whittaker and modified Hall-Littlewood functions}
\author{Aritra Bhattacharya}
\address{Beijing International Center for Mathematical Research, Peking University,
	Beijing, 100871, China}
\email{matharitra@gmail.com}
\date{\today}

\begin{document}
	
	\maketitle
	
	\begin{abstract}
		We explain the equality between the two sets of formulas for $q$-Whittaker functions and modified Hall-Littlewood functions obtained by Haglund, Haiman and Loehr - the Inv formula and Ayyer, Mandelshtam and Martin - the Quinv formula by use of weighted path symmetric functions introduced by Carlsson and Mellit.
	\end{abstract}
	
	\section{Introduction}
	The $q$-Whittaker functions $W_\lambda(q)$ and the modified Hall-Littlewood functions $\widetilde{H}_\lambda(q)$ are symmetric functions depending on a parameter $q$. They are respectively the lowest degree term and the highest degree term with respect to the parameter $t$ in the modified Macdonald symmetric functions $\widetilde{H}_\lambda(q,t)$: \begin{equation}
		\widetilde{H}_\lambda(q,t) = W_\lambda(q)\cdot t^{n(\lambda)}+\ldots+\widetilde{H}_{\lambda'}(q)\cdot t^{0}.
	\end{equation}
	  The combinatorial formulas for monomial expansions of $\widetilde{H}_\lambda(q,t)$ given by Haglund-Haiman-Loehr in \cite{HHL-I} and Ayyer-Mandelshtam-Martin in \cite{AMM-I} gives two different sets of formulas for each of the symmetric functions mentioned above. In this note we give an explanation of the fact that the two sets of formulas for the $q$-Whittaker functions and the modified Hall-Littlewood functions agree, through the use of weighted Dyck path symmetric functions introduced by Carlsson and Mellit in their proof of the Shuffle conjecture \cite{Carlsson-Mellit-Shuffle}.  Following their Example 3.10, for a partition $\lambda$, we define a Dyck path $\pi^\Inv_\lambda$ that captures the Haglund-Haiman-Loehr formulas with the $q$-weight given by the $\Inv$ statistic, \eqref{eq:HHL}. Analogously, we define another path $\pi^\Quinv_\lambda$ that captures the Ayyer-Mandelshtam-Martin formulas with the $q$-weight given by the $\Quinv$ statistic, \eqref{eq:AMM}. \autoref{res:pathInvQuinvrel} shows that these two paths on a partition $\lambda$ are related by the maps $\zeta, \zeta^{-1}$ and the reversal maps. The maps $\zeta$ and $\zeta^{-1}$ are fundamental in the theory around Dyck paths and related symmetric functions (see for example \cite{HaglundqtCatbook}). The equality of the two sets of formulas for $q$-Whittaker and modified Hall-Littlewood functions in \autoref{th:mainth} then follow from simple transformation rules of the path symmetric functions. An alternate proof is obtained in \cite{BRVqWhittaker}.
	  
	  We recall the notations related to Dyck paths, and refer the reader to \cite{MacMainBook} and \cite{HHL-I} for undefined terminology concerning partitions and symmetric functions.
	 
	\section{Main Content}
	A Dyck path of semilength $n$ is a lattice path from $(0,0)$ to $(n,n)$ consisting of unit length north steps $N$ and unit length east steps $E$ such that the path always stays weakly above the diagonal $x=y$. We write the cell co-ordinate of each box in the $n\times n$ grid from $(0,0)$ to $(n,n)$ inside $\ZZ_{\geq 0} \times \ZZ_{\geq 0}$ as the co-ordinate of its north-east corner. $\Area(\pi)$ is the set of cells below $\pi$ above the diagonal. The Dyck path $\pi$ is uniquely determined by $\Area(\pi)$. Denote by $x_i(\pi)$ the $x$-co-ordinate of the cell immediately right of the $i$-th $N$ step. The set of corners $c(\pi)$ of $\pi$ are the cells which are above the path but both its eastern and southern neighbors are below the path. \autoref{fig:pi} illustrates these notions.
	
	
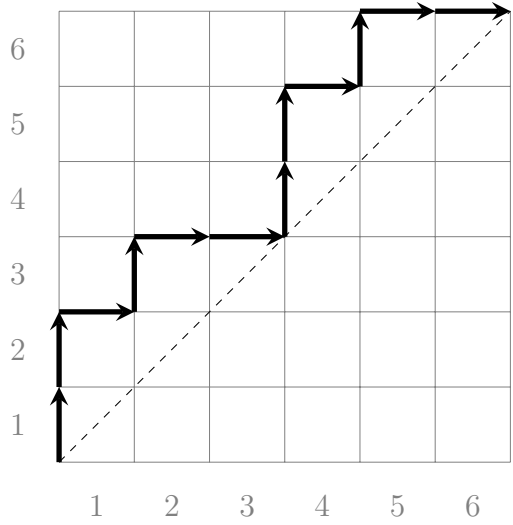
\begin{figure}[h]
		\begin{tikzpicture}
			\NEpath{0,0}{6}{6}{1,1,0,1,0,0,1,1,0,1,0,0};
		\end{tikzpicture}
		\begin{align*}
			&\Area(\pi) = \{(1,2),(2,3),(4,5),(5,6)\}, \qquad 
			c(\pi) = \{(1,3),(3,4),(4,6)\},
			\\
			&\qquad x_1(\pi) = 1, x_2(\pi) = 1, x_3(\pi) = 2, x_4(\pi) = 4, x_5(\pi) = 4, x_6(\pi) = 5.
		\end{align*}
		\caption{Example of a Dyck path}
		\label{fig:pi}
	\end{figure}
	
%
	
	The \textit{reversal map} $\rev$ on the set of Dyck paths is defined by reading the Dyck path from right to left and interchanging the $E$ and $N$ steps. \autoref{fig:reveg} gives an example. If $\pi$ has semilength $n$ then \begin{equation}\label{eq:Arearev}
		\Area(\rev(\pi)) = \{(n+1-j,n+1-i)\,|\, (i,j) \in \Area(\pi)\}.
	\end{equation}
	
	\begin{figure}[h]
		\centering
		\begin{minipage}{0.45 \textwidth} 
		\begin{tikzpicture}
			\NEpath{0,0}{4}{4}{1,0,1,1,0,1,0,0};
		\end{tikzpicture}
		\subcaption{$\pi$}
		\end{minipage}
		\begin{minipage}{0.45 \textwidth}
			\begin{tikzpicture}
				\NEpath{0,0}{4}{4}{1,1,0,1,0,0,1,0};
			\end{tikzpicture}
			\subcaption{$\rev(\pi)$}
		\end{minipage}
	\caption{}
	\label{fig:reveg}
	\end{figure}
	
	Now we recall the definition of $\zeta$ map on the set of all Dyck paths defined in \cite[\S 2.4]{Haglund-Xin-shuffle}. First, write a \textit{reading label} on the cell to the right of each $N$ step of $\pi$ as follows: label the $N$ steps in the lowest diagonal from left to right first, then label the $N$ steps in second lowest diagonal from left to right and so on. The box with reading label $i$  \textit{dinv-attacks} the box with reading label $j$ if $i<j$ and they are in the same diagonal or the box with label $j$ is in one diagonal above and to the left of the box with label $i$. Write $\boxed{i} \rightarrow \boxed{j}$ if in the reading label, the box labelled $i$ dinv-attacks the box labelled $j$. Construct the path $\zeta(\pi)$ such that \begin{equation}\label{eq:zetadef}
		\Area(\zeta(\pi)) = \{(i,j)\,|\, \boxed{i} \rightarrow \boxed{j} \hbox{ in } \pi \}.
	\end{equation} \autoref{fig:zetamapeg} gives an example. The reading label of $\pi$ determines a permutation $\sigma_\pi$ obtained by reading the labels of each column from bottom to top starting with the leftmost column.
	
	\begin{figure}[h]
		\centering
		\begin{minipage}{0.45 \textwidth}
			\begin{tikzpicture}
				\NEpath{0,0}{6}{6}{1,1,0,1,0,0,1,1,0,1,0,0};
				\node at (0.5,0.5) {1};
				\node at (0.5,1.5) {3};
				\node at (1.5,2.5) {4};
				\node at (3.5,3.5) {2};
				\node at (3.5,4.5) {5};
				\node at (4.5,5.5) {6};
			\end{tikzpicture}
			\subcaption{$\pi$ with reading labels, \\ $\sigma_\pi = 134256$}
		\end{minipage}
		\begin{minipage}{0.45 \textwidth}
		\begin{tikzpicture}
			\NEpath{0,0}{6}{6}{1,1,0,1,1,0,1,1,0,0,0,0};
		\end{tikzpicture}
		\subcaption{$\zeta(\pi)$ \\ $c(\zeta(\pi)) = \{(1,3),(2,5)\}$}
		\end{minipage}
		\caption{}
		\label{fig:zetamapeg}
	\end{figure}
	

	For a word $w \in \ZZ_{>0}^n$, let  \begin{equation}
		\inv(\pi,w) = \# \{(i,j) \in \Area(\pi)\,|\, w_i>w_j \},
	\end{equation} and \begin{equation}
	\dinv(\pi,w) = \#\{(i,j)\,|\,\boxed{i}\rightarrow\boxed{j} \hbox{ and } w_i>w_j\}.
	\end{equation}
	
	Let $\mathbb{K}$ be a ring containing $q$ and for a weight function $\wt : c(\pi) \to \mathbb{K}$, \cite[(3.4)]{Carlsson-Mellit-Shuffle} defines the weighted characteristic function $$ \chi(\pi,q,\wt) = \sum_{w \in \ZZ^n_{>0}} q^{\inv(\pi,w)}\bigg( \prod_{\substack{(i,j) \in c(\pi) \\ w_i \leq w_j}} \wt(i,j)  \bigg) x_w,$$ where for $w = (w_1,\ldots,w_n) \in \ZZ_{>0}^n$, $x_w = x_{w_1}\ldots x_{w_n}$. \cite[Proposition 3.7]{Carlsson-Mellit-Shuffle} says that $\chi(\pi,q,\wt)$ is a symmetric function with coefficients in $\mathbb{K}$. 

In particular, we will focus on the case when $\wt$ is the constant function $t$: \begin{equation}
	\chi(\pi,q,t) = \sum_{w \in \ZZ_{>0}^n} q^{\inv(\pi,w)}t^{\#\{(i,j)\in c(\pi)\,|\, w_i \leq w_j\}}x_w.
\end{equation} 

Let \begin{align}
	\WP_\pi(>) &= \{w \in \ZZ_{>0}^n\,|\, w_j > w_{j+1} \hbox{ whenever } x_j(\pi) = x_{j+1}(\pi)\},
	\\
	\WP'_\pi(>) &= \{w \in \ZZ_{>0}^n \,|\, w_i > w_j \hbox{ whenever } (i,j)\in c(\pi) \}, 
	\\
	\WP_\pi(\leq) &= \{w \in \ZZ_{>0}^n\,|\, w_j \leq w_{j+1} \hbox{ whenever } x_j(\pi) = x_{j+1}(\pi)\},
	\\
	\WP'_\pi(\leq) &= \{w \in \ZZ_{>0}^n \,|\, w_i \leq w_j \hbox{ whenever } (i,j)\in c(\pi) \}.
\end{align}

We picture elements of $\WP_\pi(>)$ (resp. $\WP_\pi(\leq)$) as fillings with positive integers on the boxes immediately right of every north step in the Dyck path $\pi$, such that in every column the entries decrease (resp. weakly increase) from bottom to top.  Similarly, the elements of $\WP_\pi'(>)$ (resp. $\WP_\pi'(\leq)$) are fillings with positive integers on the boxes on the diagonal $x=y$, such that if $(i,j)$ is a corner then the entry in the $i$th column is strictly bigger (resp. weakly smaller) than the entry in the $j$th row. 
	
Then we can extract the lowest and highest $t$-degree terms in $\chi(\pi,q,t)$ as \begin{equation}
	\chi(\pi,q,0) = \sum_{\WP'_\pi(>)} q^{\inv(\pi,w)} x_w \qquad \hbox{ and } \qquad \chi(\pi,q,t)\big|_{t^{\# c(\pi)}} = \sum_{\WP'_\pi(\leq)} q^{\inv(\pi,w)} x_w.
\end{equation}

Define \begin{equation}\label{eq:barchidef}
	\overline{\chi}(\pi,q,t) = \chi(\zeta(\pi),q,t).
\end{equation}

\cite[\S 2.4]{Carlsson-Mellit-Shuffle}, \cite[\S 2.4]{Haglund-Xin-shuffle} says
$$ \inv(\zeta(\pi),w\circ (\sigma^\pi)^{-1}) = \dinv(\pi,w) $$ and \cite[Proposition 2.2]{Haglund-Xin-shuffle} says
\begin{equation}\label{eq:czetapi}
	c(\zeta(\pi)) = \{(\sigma^\pi_r,\sigma^\pi_{r+1})\,|\, 1 \leq r <n , x_r(\pi) = x_{r+1}(\pi)\}.
\end{equation}
Then \begin{align*}
	\inv(\zeta(\pi),w) = \inv(\zeta(\pi),w\circ \sigma^\pi\circ(\sigma^\pi)^{-1} \ ) = \dinv(\pi,w \circ \sigma^\pi),
\end{align*}
and \begin{align*}
	\{(i,j) \in c(\zeta(\pi))\,|\, w_i \leq w_j\} = \{r \in [n-1]\,|\, x_r(\pi) = x_{r+1}(\pi) \hbox{ and } w_{\sigma^\pi_r} \leq w_{\sigma^\pi_{r+1}} \}.
\end{align*} Hence,
\begin{align}
	\overline{\chi}(\pi,q,t) &= \chi(\zeta(\pi),q,t) = \sum_{w \in \ZZ_{>0}^n} q^{\inv(\zeta(\pi),w)} t^{\# \{(i,j) \in c(\zeta(\pi))\,|\, w_i \leq w_j \}} x_{w}
	\nonumber
	\\
	&= \sum_{w \in \ZZ_{>0}^n} q^{\dinv(\pi,w\circ \sigma^\pi)} t^{\# \{r\,|\, x_r(\pi) = x_{r+1}(\pi)\hbox{ and } w_{\sigma^\pi_r} \leq w_{\sigma^\pi_{r+1}} \}} x_{w \circ \sigma^\pi}
	\\
	&= \sum_{w \in \ZZ_{>0}^n} q^{\dinv(\pi,w)} t^{\# \{r\,|\, x_r(\pi) = x_{r+1}(\pi)\hbox{ and } w_{r} \leq w_{{r+1}} \}} x_w.
	\nonumber
\end{align}

In particular we can extract the lowest and highest $t$-degree terms in $\overline{\chi}(\pi,q,t)$ as
\begin{equation}\label{eq:barchilowhigh}
	\overline{\chi}(\pi,q,0) = \sum_{w \in \WP_\pi(>)} q^{\dinv(\pi,w)}x_w \qquad \hbox{ and } \qquad \overline{\chi}(\pi,q,t)\big|_{t^{\# c(\zeta(\pi))}} = \sum_{w \in \WP_\pi(\leq)} q^{\dinv(\pi,w)}x_w.
\end{equation}

Let $\Par$ denote the set of integer partitions. For $\lambda \in \Par$ with $\ell(\lambda) = \ell$, define \begin{equation}
	\pi_\lambda^\Inv = N^{\lambda_\ell}\cdot(EN)^{\lambda_\ell}  N^{\lambda_{\ell-1}-\lambda_{\ell}}\cdot(EN)^{\lambda_{\ell-1}} N^{\lambda_{\ell-2}-\lambda_{\ell-1}} \cdot \ldots \cdot (EN)^{\lambda_2}  N^{\lambda_1-\lambda_2} \cdot E^{\lambda_1},
\end{equation} and  \begin{equation}
	\pi_\lambda^{\Quinv} = N^{\lambda_\ell}\cdot N^{\lambda_{\ell-1}-\lambda_{\ell}}(EN)^{\lambda_\ell}  \cdot N^{\lambda_{\ell-2}-\lambda_{\ell-1}}(EN)^{\lambda_{\ell-1}}  \cdot \ldots \cdot N^{\lambda_1-\lambda_2}(EN)^{\lambda_2} \cdot  E^{\lambda_1}.
\end{equation}  \autoref{fig:pilaminvandquinveg} illustrates the two paths when $\lambda = (3,2)$.

\begin{figure}[h]
	\centering
	\begin{minipage}{0.45 \textwidth}
		\begin{tikzpicture}
			\NEpath{0,0}{5}{5}{1,1,0,1,0,1,1,0,0,0}; 
		\end{tikzpicture}
		\subcaption{$\pi_{(3,2)}^\Inv$}
	\end{minipage}
	\hfill
	\begin{minipage}{0.45 \textwidth}
		\begin{tikzpicture}
			\NEpath{0,0}{5}{5}{1,1,1,0,1,0,1,0,0,0};
		\end{tikzpicture}
		\subcaption{$\pi_{(3,2)}^\Quinv$}
	\end{minipage}
	\caption{}
	\label{fig:pilaminvandquinveg}
\end{figure}

 For $\lambda \in \Par$, define the \textit{inversion reading order} to be the ordering of the boxes of $\lambda$ with reading left-to-right within each row and bottom to top among rows. The \textit{quinversion reading order} is the ordering of the boxes with reading right-to-left within each row and bottom to top among rows. \autoref{fig:invandquinvread} illustrates the two reading orders for $\lambda = (3,2)$.

\begin{figure}[h]
	\centering
	\begin{minipage}{0.45 \textwidth}
		\begin{align*}
			\ytableaushort{345,12}
		\end{align*}
		\subcaption{inversion reading order of $(3,2)$}
	\end{minipage}
	\hfill
	\begin{minipage}{0.45 \textwidth}
		\begin{align*}
			\ytableaushort{543,21}
		\end{align*}
		\subcaption{quinversion reading order of $(3,2)$}
	\end{minipage}
	\caption{}
	\label{fig:invandquinvread}
\end{figure}

A box \textit{inv-attacks} all the boxes in the same row to its right and all the boxes in the row above to its left. A box \textit{quinv-attacks} all the boxes in the same row to its left and all the boxes in the row above to its right. \autoref{fig:invandquinvattack} illustrates the inv- and quinv-attack relations.

\begin{figure}[h]
	\centering
	\begin{minipage}{0.45 \textwidth}
		\begin{align*}
			\ydiagram[*(gray)]{2,3+3}*[*(white) u]{0,2+1}*[*(white)]{7,6}
		\end{align*}	
		\subcaption{boxes inv-attacked by $u$}
	\end{minipage}
	\hfill
	\begin{minipage}{0.45 \textwidth}
		\begin{align*}
			\ydiagram[*(gray)]{3+4,2}*[*(white) u]{0,2+1}*[*(white)]{7,6}
		\end{align*}
		\subcaption{boxes quinv-attacked by $u$}		
	\end{minipage}
	\caption{}
	\label{fig:invandquinvattack}
\end{figure}

%

Let $\Inv(\lambda)$ (resp. $\Quinv(\lambda)$) be the set of pairs $(i,j)$ such that the $i$th box attacks the $j$th box in the inversion (quinversion) reading order of $\lambda$. 

If the box labelled $j$ in the inversion reading order (resp. quinversion reading order) is immediately above the box labelled $i$ we write $j = \up^\Inv(i)$ (resp. $j = \up^\Quinv(i)$).

 \begin{lemma}\label{lem:AreaInvQuinv}
	Let $\lambda \in \Par$. Then \begin{equation}
		\Area(\pi_\lambda^\Inv) = \Inv(\lambda) \qquad \hbox{ and }\qquad \Area(\pi_\lambda^\Quinv) = \Quinv(\lambda),
	\end{equation}
	and \begin{equation}
		c(\pi_\lambda^\Inv) = \{(i,\up^\Inv(i))\} \qquad \hbox{ and }\qquad c(\pi_\lambda^\Quinv) = \{(i,\up^\Quinv(i))\}.
	\end{equation}
\end{lemma}

\begin{proof}
	\item[(1)]
%
	Let $\pi$ be a Dyck path of semilength $n$.
	Let the height of a $E$ step be the number of $N$ steps before it. If the height of the $i$th $E$ step is $h_i$ then $$ \Area(\pi) = \{(i,j)| 1 \leq i \leq n, i<j\leq h_i\}.$$ Then the explicit expressions for the paths gives the first statement.

	\item[(2)] Corners of a Dyck path $\pi$ correspond to every occurrence of $EN$ in $\pi$, such a corner has $x$-co-ordinate equal to the number of $E$ steps before it $+1$ (since our co-ordinates start from $1$) and $y$-co-ordinate equal to the number of $N$ steps before it $+1$. Then the explicit expressions for $\pi_\lambda^\Inv$ and $\pi_\lambda^\Quinv$ gives the second statement.
\end{proof}

In particular, this implies \begin{equation}\label{eq:cornercard}
	\# c(\pi_\lambda^\Inv) = \# c(\pi_\lambda^\Quinv) = \sum_{i = 1}^{\lambda_1} (\lambda'_i-1) = |\lambda|- \lambda_1.
\end{equation}

The $q$-Whittaker functions $W_\lambda(q)$ and the modified Hall-Littlewood functions $\widetilde{H}_{\lambda'}(q)$ are the highest and lowest $t$-degree terms of the modified Macdonald functions $H_\lambda(q,t)$, where we have \cite[(3.1)]{BergeronqWhittaker} \begin{equation}\label{eq:modMac}
H_\lambda(q,t) = W_\lambda(q)\cdot t^{0}+\ldots+\widetilde{H}_{\lambda'}(q)\cdot t^{n(\lambda)}.
\end{equation} Here $\lambda'$ denotes the conjugate partition. For definitions of modified Macdonald functions and $\arm, \leg, \coleg$ used below we refer the reader to \cite{HHL-I}. Note that here we use $H_\lambda(q,t)$ without the $\widetilde{\cdot}$, which is obtained from the one used by \cite{HHL-I} and \cite{AMM-I} by $$ H_\lambda(q,t) = t^{n(\lambda)} \widetilde{H}_\lambda(q,t^{-1}).$$ There seems to be some difference in conventions in the literature, some authors write $\widetilde{H}_\lambda(q)$ in place of our $\widetilde{H}_{\lambda'}(q)$ above. We use the definition that $ \widetilde{H}_\lambda(q) = \widetilde{H}_\lambda(q,0)$, together with the identity $$\widetilde{H}_{\lambda}(q,t) = \widetilde{H}_{\lambda'}(t,q)$$ explains \eqref{eq:modMac}.


Extracting the lowest and highest $t$ degree terms in $H_\lambda(q,t)$ from \cite{HHL-I} and \cite{AMM-I} formulas corresponds to taking fillings where $\maj$ is maximum or descent set is full and $\maj$ minimum or descent set empty respectively. Then \cite[Theorem 2.2]{HHL-I} says \begin{equation}\label{eq:HHL}
	W_\lambda(q) = q^{-\alpha_\Inv(\lambda)}  \chi(\pi_\lambda^\Inv,q,0) \qquad \hbox{ and } \qquad  \widetilde{H}_{\lambda'}(q) = \chi(\pi_\lambda^\Inv,q,t)\big|_{t^{\# c(\pi_\lambda^\Inv)}} ,
\end{equation} where \begin{equation}\label{eq:alphaInv}
\alpha_\Inv(\lambda)= \sum\limits_{\substack{u \in \lambda \\ \coleg(u) \neq 0}}\arm(u).
\end{equation} \cite[Theorem 2.6]{AMM-I} says \begin{equation}\label{eq:AMM}
W_\lambda(q) = q^{-\alpha_\Quinv(\lambda)} \chi(\pi_\lambda^\Quinv,q,0) \qquad \hbox{ and } \qquad  \widetilde{H}_{\lambda'}(q) = \chi(\pi_\lambda^\Quinv,q,t)\big|_{t^{\# c(\pi_\lambda^\Quinv)}},
\end{equation} where \begin{equation}\label{eq:alphaQuinv}
\alpha_\Quinv(\lambda) = \sum\limits_{\substack{u \in \lambda \\ \leg(u) \neq 0}} \arm(u) .
\end{equation}

In particular, both the $\Inv$ and $\Quinv$ statistics on fillings of partitions can be realized as $\inv$ statistic on words on Dyck paths.

The main purpose of this article is to prove the equality between the right hand sides of \eqref{eq:HHL} and \eqref{eq:AMM}, this is achieved in \autoref{th:mainth}.

For a partition $\lambda$, the symmetric functions $\chi(\pi_\lambda^\Inv,q,t)$ and $\chi(\pi^\Quinv_\lambda,q,t)$ are two families of symmetric functions that has as its lowest $t$-degree term and highest $t$-degree terms the $q$-Whittaker functions and the modified Hall-Littlewood functions respectively: \begin{align*}
	\chi(\pi^\Inv_\lambda,q,t) &= q^{\alpha_\Inv(\lambda)}W_\lambda(q)\cdot t^0+\ldots+\widetilde{H}_{\lambda'}(q)\cdot t^{\# c(\pi_\lambda^\Inv)}, 
	\\
	\chi(\pi^\Quinv_\lambda,q,t) &= q^{\alpha_\Quinv(\lambda)}W_\lambda(q)\cdot t^0+\ldots+\widetilde{H}_{\lambda'}(q)\cdot t^{\# c(\pi_\lambda^\Quinv)}.
\end{align*} This is analogous to \eqref{eq:modMac} for the modified Macdonald functions.

For use in \autoref{th:mainth} we calculate the difference between $\alpha_\Quinv$ and $\alpha_\Inv$ below.

\begin{lemma}\label{lem:alpQ-alpI}
	Let $\lambda \in \Par$ and $m'_i = m_i(\lambda')$ be the multiplicity of $i$ in the conjugate partition $\lambda'$. Then  $$\alpha_\Quinv(\lambda) - \alpha_\Inv(\lambda) = \sum_{i>j} m'_im'_j.$$
\end{lemma}
\begin{proof}
	Since $$\alpha_\Quinv(\lambda) = \sum_{u \in \lambda} \arm(u) - \sum_{\substack{u \in \lambda \\ \leg(u) = 0}} \arm(u),$$ and $$ \alpha_\Inv(\lambda) = \sum_{u \in \lambda} \arm(u) - \sum_{\substack{u \in \lambda \\ \coleg(u) = 0}} \arm(u),$$ so \begin{align*}
		\alpha_\Quinv(\lambda) - \alpha_\Inv(\lambda) &= \sum_{\substack{u \in \lambda \\ \coleg(u) = 0}} \arm(u) - \sum_{\substack{u \in \lambda \\ \leg(u) = 0}} \arm(u) = \binom{\lambda_1}{2} - \sum_{i \geq 1} \binom{m'_i}{2} 
		\\
		&= \binom{\sum_{i \geq 1} m'_i}{2} - \sum_{i \geq 1} \binom{m'_i}{2}  
		\\
		&= \dfrac{1}{2} \cdot \bigg((\sum_{i \geq 1}m'_i)(\sum_{i \geq 1}m'_i - 1) - \sum_{i \geq 1}m'_i(m'_i-1)\bigg) 
		\\
		&= \sum_{i>j} m'_im'_j.
	\end{align*}
\end{proof}

The next lemma shows that we can view the right hand side from \autoref{lem:alpQ-alpI} as the length of the smallest permutation that makes $\lambda'$ into an antipartition.

\begin{lemma}\label{lem:revlength}
	Let $\lambda = (n^{m_n},\ldots,1^{m_1}) \in \Par$. The smallest permutation that takes $\lambda$ to $\rev(\lambda)$ has length =  $\sum_{i>j} m_im_j$. 
\end{lemma}

\begin{proof}
	Recall that length of a permutation is the smallest number of simple transpositions required to write it. Starting from $\lambda$, first we move the last $n$ all the way across $n-1,\ldots,1$s. This requires $m_{n-1}+\ldots+m_1$ many simple transpositions. Then we move the second last $n$ all the way across $n-1,\ldots,1$s, this also requires $m_{n-1}+\ldots+m_1$ many simple transpositions. So to move all the $n$ across all the $n-1,\ldots,1$s we require $m_n(m_{n-1}+\ldots+m_1)$ many simple transpositions. Then we move all the $n-1$s across all the $n-2,\ldots,1$s and so on. Finally we move all the $2$s across all the $1$s. So the whole procedure involves \begin{align*}
		m_n(m_{n-1}+\ldots+m_1) + m_{n-1}(m_{n-2}+\ldots+m_1)+\ldots+m_2(m_1) = \sum_{i>j} m_im_j
	\end{align*} many simple transpositions.
\end{proof}

Now we will analyze the relation between $\pi^\Inv_\lambda$ and $\pi^\Quinv_\lambda$ using the maps $\zeta, \zeta^{-1}$ and $\rev$.

\begin{proposition}\label{prop:revzetapilam}
	For $\lambda \in \Par$ with $\lambda' = (\lambda'_1,\ldots,\lambda'_k)$, $k = \lambda_1$, let \begin{equation}
		\pi_{\lambda} = N^{\lambda'_k}E^{\lambda'_k} \cdot \ldots \cdot N^{\lambda'_1}E^{\lambda'_1}.
	\end{equation}. Then \begin{equation}
	\rev \circ \zeta(\pi_\lambda) = \pi^\Inv_\lambda \qquad \hbox{ and } \qquad \rev \circ \zeta \circ \rev (\pi_\lambda) = \pi^\Quinv_\lambda.
	\end{equation} 
\end{proposition}

\begin{proof}

	\item[(1)] $$ \boxed{|\lambda|+1-j} \rightarrow \boxed{|\lambda|+1-i} \hbox{ in } \pi_\lambda \qquad \hbox{ if and only if }\qquad (i,j) \in \Inv(\lambda).$$ 
	To see this, place the columns of $\pi_\lambda$ with reading labels in reverse order, with gravity working upwards so that we get a partition shape, and reverse the reading-label inside each column. Then the box labelled $|\lambda|+1-j$ inv-attacks the box labelled $|\lambda|+1-i$ in this new diagram if and only if $\boxed{|\lambda|+1-i}\rightarrow \boxed{|\lambda|+1-j}$ in $\pi_\lambda$. Changing each label $i$ to $|\lambda|+1-i$ produces $\lambda$ with inversion reading order. For example if $\lambda = (3,2)$, $\pi_{(3,2)}$ is pictured in \autoref{fig:pi_32andzeta}. \begin{figure}[h]
		\centering
		\begin{minipage}{0.45 \textwidth}
			\begin{tikzpicture}
				\NEpath{0,0}{5}{5}{1,0,1,1,0,0,1,1,0,0}; 
				\node at (0.5,0.5) {1};
				\node at (1.5,1.5) {2};
				\node at (1.5,2.5) {4};
				\node at (3.5,3.5) {3};
				\node at (3.5,4.5) {5};
			\end{tikzpicture}
			\subcaption{$\pi_{(3,2)}$}
		\end{minipage}
		\hfill
		\begin{minipage}{0.45 \textwidth}
			\begin{tikzpicture}
				\NEpath{0,0}{5}{5}{1,1,1,0,0,1,0,1,0,0}
			\end{tikzpicture}
			\subcaption{$\zeta(\pi_{(3,2)}) = \rev(\pi^\Inv_{(3,2)})$}
		\end{minipage}
		\caption{}
		\label{fig:pi_32andzeta}
	\end{figure} In this case placing the columns of $\pi_{(3,2)}$ in reverse order with gravity upwards produces $\ytableaushort{541,32}$. Then reversing label within each column gives $\ytableaushort{321,54}$, finally replacing each $i$ with $6-i$ gives $\ytableaushort{345,12}$, which is the inversion reading order of $(3,2)$.  
	
	Then \eqref{eq:zetadef} gives $$ \Area(\zeta(\pi_\lambda)) = \{(|\lambda|+1-j,|\lambda|+1-i)\,|\,(i,j) \in \Inv(\lambda)\}.$$ Applying \eqref{eq:Arearev} and \autoref{lem:AreaInvQuinv} gives $$ \Area(\rev \circ \zeta(\pi_\lambda)) = \Area(\pi^\Inv_\lambda),$$ which is the first statement. 

	\item[(2)] 
	$$ \boxed{|\lambda|+1-i} \rightarrow \boxed{|\lambda|+1-j} \hbox{ in } \rev(\pi_\lambda) \qquad\hbox{ if and only if }\qquad (i,j) \in \Quinv(\lambda).$$
	To see this, write the columns of $\rev(\pi_{\lambda})$ with reading labels left-to-right, with gravity upwards, so that we get a shape of a partition, and reverse the labels in each column from bottom to top. Then the box labelled $|\lambda|+1-j$ quinv-attacks the box labelled $|\lambda|+1-i$ in this new diagram if and only if $\boxed{|\lambda|+1-i}\rightarrow \boxed{|\lambda|+1-j}$ in $\pi_\lambda$. Changing each label $i$ to $|\lambda|+1-i$ produces $\lambda$ with quinversion reading order. For example if $\lambda = (3,2)$, $\pi_{(3,2)}$ is pictured in \autoref{fig:pi_32andzeta}. \begin{figure}[h]
		\centering
		\begin{minipage}{0.45 \textwidth}
			\begin{tikzpicture}
				\NEpath{0,0}{5}{5}{1,1,0,0,1,1,0,0,1,0};
				\node at (0.5,0.5) {1};
				\node at (0.5,1.5) {4};
				\node at (2.5,2.5) {2};
				\node at (2.5,3.5) {5};
				\node at (4.5,4.5) {3};
			\end{tikzpicture}
			\subcaption{$\rev(\pi_{(3,2)})$}
		\end{minipage}
		\hfill
		\begin{minipage}{0.45 \textwidth}
			\begin{tikzpicture}
				\NEpath{0,0}{5}{5}{1,1,1,0,1,0,1,0,0,0};
			\end{tikzpicture}
			\subcaption{$\zeta(\rev(\pi_{(3,2)})) = \rev(\pi_{(3,2)}^\Quinv)$}
		\end{minipage}
		\caption{}
		\label{fig:revpi32andzeta}
	\end{figure}
	 In this case placing the columns of $\pi_{(3,2)}$ in reverse order with gravity upwards produces $\ytableaushort{453,12}$. Then reversing label within each column gives $\ytableaushort{123,45}$, finally replacing each $i$ with $6-i$ gives $\ytableaushort{543,21}$, which is the quinversion reading order of $(3,2)$.  
		
	Then \eqref{eq:zetadef} gives $$ \Area(\zeta\circ\rev(\pi_\lambda)) = \{(|\lambda|+1-j,|\lambda|+1-i)\,|\,(i,j) \in \Quinv(\lambda)\}.$$ Applying \eqref{eq:Arearev} and \autoref{lem:AreaInvQuinv} gives $$ \Area(\rev \circ \zeta \circ \rev(\pi_\lambda)) = \Area(\pi^\Quinv_\lambda),$$ which is the second statement. 
	
\end{proof}


\begin{corollary}\label{res:pathInvQuinvrel}
	Let $\lambda \in \Par$. Then \begin{equation}\label{eq:pathInvQuinvrel}
		\rev\circ \zeta \circ \rev \circ \zeta^{-1} \circ \rev (\pi_\lambda^\Inv) = \pi_\lambda^\Quinv .
	\end{equation}
\end{corollary}


%

Now we will analyze the change in $\chi$ and $\overline{\chi}$ under the reversal map.

The following lemma generalizes \cite[Proposition 3.3]{Carlsson-Mellit-Shuffle}.

\begin{lemma}\label{lem:chirev}
	Let $\pi$ be a Dyck path. Then $$ \chi(\pi,q,t) = \chi(\rev(\pi),q,t).$$
\end{lemma}

\begin{proof}
	Let $\pi$ be a Dyck path of semilength $n$. Then $$ c(\rev(\pi)) = \{(n+1-j,n+1-i)\,|\,(i,j) \in c(\pi)\}.$$
	For a word $w = (w_1,\ldots,w_n) \in \ZZ^n_{>0}$ let $m = \max\{w_i|i\in [n]\}$ and $\widetilde{w}_{i} = m+1 - w_{n+1-i}$. Then $\widetilde{w} \in \ZZ_{>0}^n$. 
	If $(i,j) \in c(\pi)$ then $w_i > w_j$ if and only if $\widetilde{w}_{n+1-j} = m+1-w_j > \widetilde{w}_{n+1-i} = m+1-w_i$. 
	Therefore, $\inv(\pi,w) = \inv(\rev(\pi),\widetilde{w})$ and $x_{\widetilde{w}} = w_0(x_w)$, where $w_0$ is the longest permutation in $S_m$. Therefore, for $\alpha \in \ZZ_{>0}^n$, coefficient of $x^\alpha$ in $\chi(\pi,q,t)$ is same as coefficient of $x^{w_0(\alpha)}$ in $\chi(\rev(\pi),q,t)$, but since $\chi(\rev(\pi),q,t)$ is symmetric, it is same as coefficient of $x^{\alpha}$.   
\end{proof}

If $\pi$ is balanced, i.e, $\pi$ can be broken into a composition of blocks of the form $N^kE^k$ then we can calculate the change in the lowest and highest $t$-degree terms of $\overline{\chi}$ after a permutation of the blocks as follows.

\begin{lemma}\label{lem:barchisimpletransposition}
	Let $\pi = N^{\ell_1}E^{\ell_1}\cdot\ldots\cdot N^{\ell_n}E^{\ell_n}$. Suppose $\ell_i < \ell_{i+1}$. Let $\pi' = N^{\ell_1}E^{\ell_1}\cdot\ldots\cdot N^{\ell_{i+1}}E^{\ell_{i+1}}\cdot N^{\ell_i}E^{\ell_i}\cdot\ldots\cdot N^{\ell_n}E^{\ell_n}$ be the Dyck path with $i$th and $i+1$th blocks interchanged. Then $$ \overline{\chi}(\pi',q,0) = q\cdot \overline{\chi}(\pi,q,0), \qquad \hbox{ and } \qquad \overline{\chi}(\pi',q,t)\big|_{t^{\# c(\zeta(\pi'))}} = \overline{\chi}(\pi,q,t)\big|_{t^{\# c(\zeta(\pi))}}. $$ 
\end{lemma}
\begin{proof}
	Call a word $w = (w_1,\ldots,w_m)$ a column decreasing (resp. weakly increasing) word if $w_1>\ldots>w_m$ (resp. $w_1\leq \ldots\leq w_m$). Denote by $\len(w)$ the length $m$ of the word $w$.
	
	\item[(1)] 
	Recall from \eqref{eq:barchilowhigh} that $\overline{\chi}(\pi,q,0)$ is a sum over $\WP_\pi(>)$, which means we have a $n$-tuple of column decreasing words of lengths $\ell_1,\ldots,\ell_n$.
	
	We recall the splice operation from \cite{BRVqWhittaker} on a pair of column decreasing words $(F,G)$ where $\len(F) < \len(G)$. Let $F = (f_1,\ldots,f_a), G=(g_1,\ldots,g_b)$ with $a<b$. Let $f_0 = \infty$ and $m = \min\{j \in \{0,\ldots,a\}\,|\, f_{a-j}>g_{a-j+1}\}$. Then define $\widetilde{F} = (f_1,\ldots,f_{a-m},g_{a-m+1},\ldots,g_b)$ and $\widetilde{G} = (g_1,\ldots,g_{a-m},f_{a-m+1},\ldots,f_a)$. By definition $f_{a-m}>g_{a-m+1}$ so $\widetilde{F}$ is decreasing. If $m > 0$ then the lengths of $F$ and $G$ differ by atleast $1$, so $b \geq a-m+2$. Since $m$ was minimum, $f_{a-m+1}\leq g_{a-m+2}<g_{a-m}$ so $\widetilde{G}$ is decreasing. The operation $(F,G) \mapsto (\widetilde{F},\widetilde{G})$ therefore defines a map on a pair of column decreasing words with $\len(F) = \len(\widetilde{G}) < \len(G) = \len(\widetilde{F})$.

	Let $w \in \WP_\pi(>)$. This determines a tuple of column decreasing words $(F_1,\ldots,F_n)$. Perform the splice operation on the pair $(F_i,F_{i+1})$. Since splice does not change which diagonal a particular entry belongs to, the only difference in $\dinv$ is contributed from the two column decreasing words $F_i$ and $F_{i+1}$. Hence, it suffices to consider the case when $\pi$ has only two bounce blocks, i.e, $n = 2$, with $\ell_1 = a$ and $\ell_2 = b$ and $a<b$. We will analyze the difference in $\dinv$ from the pair $(F,G)$ and $(\widetilde{F},\widetilde{G})$. As before let $m = \min\{j \in \{0,\ldots,a\}\,|\, f_{a-j}>g_{a-j+1}\}$.   
	
	
	
	Case 1: $m = 0$. Then $\widetilde{F} = (f_1,\ldots,f_{a},g_{a+1},\ldots,g_{b})$ and $\widetilde{G} = (g_1,\ldots,g_{a})$, so one new dinv is introduced by the inequality $g_{a} > g_{a+1}$.
	
	\begin{figure}[h]
		\begin{minipage}{0.45 \textwidth}
			\begin{align*}
				\ytableausetup{boxframe=normal,boxsize=2em}
				\begin{ytableau}
					\none & g_b \\ \none & \none[\vdots] \\ \none & g_{a+1} \\ f_a & g_a \\ \none[\vdots] & \none[\vdots] \\ f_1 & g_1
				\end{ytableau}
				\qquad \qquad
				\begin{ytableau}
					g_b & \none \\ \none[\vdots] & \none \\ g_{a+1} & \none \\ f_a & g_a \\ \none[\vdots] & \none[\vdots] \\ f_1 & g_1
				\end{ytableau}
			\end{align*}
			\subcaption{$m=0$}
		\end{minipage}
		\hfill
		\begin{minipage}{0.45 \textwidth}
			\begin{align*}
			\ytableausetup{boxframe=normal,boxsize=3.5em}
			\begin{ytableau}
				\none & g_b \\ \none & \none[\vdots] \\ f_a & g_a \\ \none[\vdots] & \none[\vdots] \\ f_{a-m+1} & g_{a-m+1} \\ f_{a-m} & g_{a-m} \\ \none[\vdots] & \none[\vdots] \\ f_1 & g_1
			\end{ytableau}
			\qquad \qquad
			\begin{ytableau}
				g_b & \none \\ \none[\vdots] & \none \\ g_a & f_a \\ \none[\vdots] & \none[\vdots] \\ g_{a-m+1} & f_{a-m+1} \\ f_{a-m} & g_{a-m} \\ \none[\vdots] & \none[\vdots] \\ f_1 & g_1
			\end{ytableau}
			\end{align*}
			\subcaption{$m>0$}
		\end{minipage}
		\caption{}
	\end{figure}
		
	Case 2: $m > 0$. Since the lengths of $F$ and $G$ differ by atleast $1$, $b\geq a-m+2$. Since $m$ is minimum, $g_{a-m+i} > g_{a-m+i+1} \geq f_{a-m+i}$ for $i = 1,\ldots,m$. So $g_{a-m+i} > f_{a-m+i}$ for $i = 1,\ldots, m$ are dinvs for $(\widetilde{F},\widetilde{G})$. The dinvs $g_{a-m+i-1} > f_{a-m+i}$ for $i = 1,\ldots,m$ in $(F,G)$ are replaced by the dinvs $g_{a-m+i} > f_{a-m+i}$ for $i = 1,\ldots, m$  for $(\widetilde{F},\widetilde{G})$. The dinv $g_{a-m} > g_{a-m+1}$ for $(\widetilde{F},\widetilde{G})$ is extra.  
	
	Since both cases we have $1$ extra dinv, this proves the first statement. 
	
	\item[(2)] 
	Note that by \eqref{eq:czetapi}, $\# c(\zeta(\pi')) = \# c(\zeta(\pi)) = \# \{r\,|\, x_r(\pi) = x_{r+1}(\pi)\}$.
	Recall from \eqref{eq:barchilowhigh} that $\overline{\chi}(\pi,q,t)\big|_{t^{\# c(\zeta(\pi))}}$ is a sum over $\WP_\pi(\leq)$, which means we have a $n$-tuple of column weakly increasing words of lengths $\ell_1,\ldots,\ell_n$. 
	
	We use a variation of the above splice map. Let $F = (f_1 \leq \ldots \leq f_a), G = (g_1 \leq \ldots \leq g_b)$ with $a < b$. Let $f_0 = 0$ and $m = \min\{j\in \{0,\ldots,a\}\,|\, f_{a-j} \leq g_{a-j+1}\}$. Then define $\widetilde{F} = (f_1,\ldots,f_{a-m},g_{a-m+1},\ldots,g_b)$ and $\widetilde{G} = (g_1,\ldots,g_{a-m},f_{a-m+1},\ldots,f_a)$. By definition $f_{a-m} \leq g_{a-m+1}$ so $\widetilde{F}$ is weakly increasing. Since $m$ was minimum, $f_{a-m+1} > g_{a-m+2} \geq g_{a-m}$ so $\widetilde{G}$ is weakly increasing. The operation $(F,G) \mapsto (\widetilde{F},\widetilde{G})$ therefore defines a map on a pair of column weakly increasing words with $\len(F) = \len(\widetilde{G}) < \len(G) = \len(\widetilde{F})$. 
	
	As before, the change in dinv is only on the pair of columns on which splice is applied, and it suffices to consider $\pi$ with two bounce blocks and two column weakly increasing words $F$ and $G$ as above. 
	
	Case 1: $m = 0$. Then $\widetilde{F} = (f_1,\ldots,f_a,g_{a+1},\ldots,g_b)$ and $\widetilde{G} = (g_1,\ldots,g_a)$. Since $g_{a} \leq g_{a+1}$, no new dinv is produced.
	
	Case 2: $m>0$. Then $f_{a-m+i}>g_{a-m+i+1}$ for $i \in \{1,\ldots,m\}$ are dinvs for $(\widetilde{F},\widetilde{G})$ which are not present in $(F,G)$. If $f_{a-m+i} \leq g_{a-m+i}$ then since $G$ is weakly increasing, $f_{a-m+i} \leq g_{a-m+i+1}$, a contradiction. Hence $f_{a-m+i} > g_{a-m+i}$ for $i \in \{1,\ldots,m\}$, and these are dinvs of $(F,G)$ which are not present in $(\widetilde{F},\widetilde{G})$. Hence the total dinv is constant. 
	
	Since both cases we have no change in dinv the second statement is proved.

\end{proof}

We can now prove the equality between the right hand sides of \eqref{eq:HHL} and \eqref{eq:AMM}.

\begin{theorem}\label{th:mainth}
	Let $\lambda \in \Par$. Then
	\begin{align*}
		q^{-\alpha_\Quinv(\lambda)}\chi(\pi_\lambda^\Quinv,q,0) = q^{-\alpha_\Inv(\lambda)} \chi( \pi_\lambda^\Inv,q,0),
		\\
		\chi(\pi_\lambda^\Inv,q,t)\big|_{t^{\# c(\pi_\lambda^\Inv)}} = \chi(\pi_\lambda^\Quinv,q,t)\big|_{t^{\# c(\pi_\lambda^\Quinv)}}.
	\end{align*}
\end{theorem}

\begin{proof}
	Let $m'_i = m_i(\lambda')$ be the multiplicity of $i$ in $\lambda'$.
	
	Since by \autoref{prop:revzetapilam},
	$$\rev \circ \zeta^{-1} \circ \rev(\pi_\lambda^\Inv) = N^{\lambda'_1}E^{\lambda'_1}\cdot\ldots\cdot N^{\lambda'_m}E^{\lambda'_m},$$ then by \autoref{lem:barchisimpletransposition} and \autoref{lem:revlength}, $$ \overline{\chi}(\rev \circ \zeta^{-1} \circ \rev (\pi_\lambda^\Inv),q,0) = q^{\sum_{i>j} m'_im'_j } \cdot \overline{\chi}( \zeta^{-1} \circ \rev (\pi_\lambda^\Inv),q,0),$$ where $m'_i = m_i(\lambda')$.
Then using \autoref{res:pathInvQuinvrel}, \autoref{lem:chirev} and \eqref{eq:barchidef}, \begin{align*}
		\chi(\pi_\lambda^\Quinv,q,0) &= \chi(\rev\circ \zeta \circ \rev \circ \zeta^{-1} \circ \rev (\pi_\lambda^\Inv),q,0) = \chi(\zeta \circ \rev \circ \zeta^{-1} \circ \rev (\pi_\lambda^\Inv),q,0)  \\&= \overline{\chi}(\rev \circ \zeta^{-1} \circ \rev (\pi_\lambda^\Inv),q,0) = q^{\sum_{i>j} m'_im'_j } \cdot \overline{\chi}( \zeta^{-1} \circ \rev (\pi_\lambda^\Inv),q,0) \\&= q^{\sum_{i>j} m'_im'_j } \cdot \chi(  \rev (\pi_\lambda^\Inv),q,0) = q^{\sum_{i>j} m'_im'_j } \cdot \chi( \pi_\lambda^\Inv,q,0).
	\end{align*}
	The first statement now follows from \autoref{lem:alpQ-alpI}. 

By \eqref{eq:cornercard}, $\# c(\pi_\lambda^\Quinv) = \# c(\pi^\Inv_\lambda) = |\lambda|-\lambda_1$. Using \autoref{res:pathInvQuinvrel}, \autoref{lem:chirev}, \eqref{eq:barchidef} and \autoref{lem:barchisimpletransposition},
	\begin{align*}
		\chi(\pi_\lambda^\Quinv,q,t)\big|_{t^{|\lambda|-\lambda_1}} &= \chi(\rev\circ\zeta \circ \rev \circ \zeta^{-1} \circ \rev (\pi_\lambda^\Inv),q,t)\big|_{t^{|\lambda|-\lambda_1}} \\&= \chi(\zeta \circ \rev \circ \zeta^{-1} \circ \rev (\pi_\lambda^\Inv),q,t)\big|_{t^{|\lambda|-\lambda_1}} \\&= \overline{\chi}(\rev \circ \zeta^{-1} \circ \rev (\pi_\lambda^\Inv),q,t)\big|_{t^{|\lambda|-\lambda_1}} = \overline{\chi}( \zeta^{-1} \circ \rev (\pi_\lambda^\Inv),q,t)\big|_{t^{|\lambda|-\lambda_1}} \\&= \chi(  \rev (\pi_\lambda^\Inv),q,t)\big|_{t^{|\lambda|-\lambda_1}} = \chi( \pi_\lambda^\Inv,q,t)\big|_{t^{|\lambda|-\lambda_1}}.
	\end{align*}
\end{proof}

\section{Remarks on Schur positivity}

For a partition $\lambda$, the symmetric functions $\chi(\pi^\Inv_\lambda,q,t)$ and $\chi(\pi^\Quinv_\lambda,q,t)$ specialize to the $q$-Whittaker and modified Hall-Littlewood functions at lowest $t$-degree and highest $t$-degree respectively, this is analogous to the modified Macdonald functions $H_\lambda(q,t)$. We know that the modified Macdonald functions are Schur positive. As explained in \cite[Remark 3.6]{Carlsson-Mellit-Shuffle}, $\chi(\pi,q,0)$ and $\chi(\pi,q,1)$ are both examples of LLT polynomials, which are known to be Schur positive. Is $\chi(\pi,q,t)$ is Schur-positive for any Dyck path $\pi$? We have verified that this is the case for all Dyck paths of semilength $\leq 9$. If $\pi$ is a Dyck path of semilength $n$, then  \begin{equation}
	\chi(\pi,1,1) = s_{1}^n.
\end{equation}  Therefore, there should be statistics $\mathrm{q stat}$ and $\mathrm{t stat}$ on the set of standard Young tableaux such that $$ \chi(\pi,q,t) = \sum_{T \in \mathrm{SYT}(\lambda)} q^{\mathrm{q stat}(\pi,T)}t^{\mathrm{t stat}(\pi,T)}s_{\mathrm{shape}(T)}.$$

\begin{example}
	If $\pi = (NE)^n$, then \begin{align*}
		\chi(\pi,q,t) &= \sum_{w \in \ZZ_{>0}^n} t^{\# \{(i,j) \in c(\pi)| w_i \leq w_j\}}x_w = \sum_{w \in \ZZ_{>0}^n} t^{\# \{i \in [n-1] | w_i \leq w_{i+1}\}}x_w \\& = \sum_{w \in \ZZ_{>0}^n} t^{n-1-\# \{i \in [n-1] | w_i > w_{i+1}\}}x_w = \sum_{w \in \ZZ_{>0}^n} t^{n-1 - \mathrm{des}(w)} x_w.
	\end{align*}
	
	For a standard Young tableau $T$, let $\mathrm{des}(T)$ be the number of $i$ such that $i+1$ occurs in a row strictly below $i$ in $T$ (English notation). Then using RSK and row bumping lemma \cite{Fulton_ytableauxbook} we get
	$$ \sum_{w \in \ZZ_{>0}^n} t^{\mathrm{des}(w)}x_w = \sum_{\lambda \vdash n}s_\lambda \sum_{T \in \mathrm{SYT}(\lambda)} t^{\mathrm{des}(T)}.$$
	
	Then $$ \chi(\pi,q,t) = \sum_{\lambda \vdash n}s_\lambda \sum_{T \in \mathrm{SYT}(\lambda)} t^{n-1-\mathrm{des}(T)}.$$
\end{example}

\begin{example}
	If $\pi = N^nE^n$, then \begin{align*}
		\chi(\pi,q,t) &= \sum_{w \in \ZZ_{>0}^n} q^{\# \{1 \leq i < j \leq n \, |\, w_i>w_j\}}x_w = \sum_{w \in \ZZ_{>0}^n} q^{\inv(w)}x_w \\&= \widetilde{H}_{(n)}(q,t) = \dfrac{h_n\bigg[\dfrac{X}{1-q}\bigg]}{h_n\bigg[\dfrac{1}{1-q}\bigg]} = \sum_{\lambda \vdash n}s_\lambda \sum_{T \in \mathrm{SYT}(\lambda)} q^{\maj(T)},
	\end{align*} where $\maj(T) = \sum_{i \in \mathrm{des}(T)}i$. 
\end{example}	


Now we calculate the Schur coefficients corresponding to single row and single column partitions. Firstly, \begin{equation}
	\langle \chi(\pi,q,t),s_{(n)} \rangle = \langle \chi(\pi,q,t),h_n \rangle = \chi(\pi,q,t)\big|_{m_{(n)}} = t^{\# c(\pi)}.
\end{equation}

If $S \subseteq c(\pi)$, let $\pi_S$ denote the path obtained by flipping the corners that are in $S$. Then \cite[(3.7)]{Carlsson-Mellit-Shuffle} can be generalized to get \begin{equation}
	\chi(\pi,q,t) = \sum_{S \subseteq c(\pi)} \bigg(\dfrac{qt-1}{q-1}\bigg)^{\# c(\pi)-S}\bigg(\dfrac{1-t}{q-1}\bigg)^{\# S} \chi(\pi_S,q,1).
\end{equation} 

For a symmetric function $F$ with coefficients in $\QQ(q,t)$ let $\overline{\omega}F[X] = \overline{F}[-X]$ in plethystic notation, where $\overline{F}$ is obtained by applying the $\QQ$-algebra homomorphism on coefficients: $\overline{q} = q^{-1}, \overline{t} = t^{-1}$. Then \cite[Proposition 3.4]{Carlsson-Mellit-Shuffle} gives for a Dyck path $\pi$ of semilength $n$, \begin{equation}
\overline{\omega}(\chi(\pi,q,1)) = (-1)^n q^{-\# \Area(\pi)} \chi(\pi,q,1).
\end{equation}

Combining the last two equations and using $$ \# \Area(\pi) = \# \Area(\pi_S) - \# S \qquad \hbox{ for all } S \subseteq c(\pi),$$  we get \begin{align}
	\overline{\omega}(\chi(\pi,q,t)) &= \sum_{S \subseteq c(\pi)} \bigg(\dfrac{q^{-1}t^{-1}-1}{q^{-1}-1}\bigg)^{\# c(\pi)-S}\bigg(\dfrac{1-t^{-1}}{q^{-1}-1}\bigg)^{\# S} \overline{\omega}(\chi(\pi_S,q,1))
	\nonumber
	\\
	&= \sum_{S \subseteq c(\pi)}\bigg(\dfrac{qt-1}{q-1}\bigg)^{\# c(\pi)-S}\bigg(\dfrac{1-t}{q-1}\bigg)^{\# S} t^{-c(\pi)} q^{\# S} q^{-\# \Area(\pi_S)}(-1)^n \chi(\pi_S,q,1)
	\nonumber
	\\
	&= q^{-\# \Area(\pi)}t^{-\# c(\pi)} (-1)^n \sum_{S \subseteq c(\pi)} \bigg(\dfrac{qt-1}{q-1}\bigg)^{\# c(\pi)-S}\bigg(\dfrac{1-t}{q-1}\bigg)^{\# S} \chi(\pi_S,q,1)
	\nonumber
	\\
	&= q^{-\# \Area(\pi)}t^{-\# c(\pi)} (-1)^n \chi(\pi,q,t). 
\end{align}

Then \begin{align}
	\langle \chi(\pi,q,t), s_{(1^n)} \rangle &= \overline{\langle \overline{\omega}(\chi(\pi,q,t)), \overline{\omega}(s_{(1^n)}) \rangle } 
	\nonumber
	\\
	&= \overline{\langle q^{-\# \Area(\pi)}t^{-\# c(\pi)} (-1)^n \chi(\pi,q,1), (-1)^ns_{(n)} \rangle } = q^{\# \Area(\pi)}.
\end{align}
	

%
%
%
		
%
%

\bibliographystyle{alpha}
\bibliography{pathsym}

%

\end{document}